\documentclass[11pt,a4paper,reqno]{amsart}

\usepackage{a4wide,color}

\usepackage[english]{babel}
\usepackage[utf8]{inputenc}

\usepackage{bbm,amsmath,amsthm,epsfig,latexsym,marvosym}
\usepackage{amsfonts}
\usepackage{bigints}
\usepackage{amsmath}
\usepackage{amssymb}
\usepackage{subfig} 
\usepackage{microtype}
\usepackage{enumitem}
\allowdisplaybreaks
\usepackage{esint,xfrac}
\usepackage[colorlinks=true,linkcolor=blue,citecolor=red]{hyperref}
\usepackage[capitalise]{cleveref}
\newcommand{\R}{\mathbb{R}}
\newcommand{\C}{\mathbb{C}}
\newcommand{\Ld}{\mathcal{L}^{d}}

\newcommand{\A}{\mathcal{A}}
\newcommand{\aA}{\mathbb{A}}

\newcommand{\ap}{\mathrm{ap}}
\newcommand{\supp}{\operatorname{spt}}
\newcommand{\dist}{\operatorname{dist}}
\newcommand{\eps}{\varepsilon}

\newcommand{\BV}{\mathrm{BV}}
\newcommand{\BD}{\mathrm{BD}}
\newcommand{\LL}{\mathrm{L}}
\def\ca{\mathbbmss{1}}

\newtheorem{theorem}{Theorem}[section]
\newtheorem{proposition}[theorem]{Proposition}
\newtheorem{lemma}[theorem]{Lemma}

\newtheorem{remark}[theorem]{Remark}
\newtheorem{definition}[theorem]{Definition}

\title{From \(\BV^\mathcal A\) to \(\BV\): An endpoint Korn estimate}

\author{Marco Caroccia}
\address{DIpartimento di MAtematica e Informatica \textit{Ulisse Dini}}
\email{marco.caroccia@unifi.it}

\makeindex
\begin{document}
\maketitle

\begin{abstract}
In this short note we prove that, given a $\C$-elliptic operator $\A$ and a map
$u\in \BV^{\mathcal{A}}(\Omega;V)$ satisfying
\(
\nabla_{\mathrm{ap}}u\in
L^1(\Omega;V\otimes\R^d),
\)
then $u\in \BV(\Omega;V)$. The result is quantitative and follows from
the Korn-type estimate
\[
|Du|(\Omega)
\leq
C_{d,\A}\left(
\|\nabla_{\mathrm{ap}}u\|_{\mathrm{L}^1(\Omega)}
+
|\A u|(\Omega)
\right),
\]
valid for every such $u\in \BV^\mathcal{A}(\Omega;V)$. As a consequence, we obtain the characterization
\[
\BV^\A(\Omega;V)\setminus \BV(\Omega;V)
=
\left\{
u\in \BV^\A(\Omega;V):
\nabla_{\mathrm{ap}}u\notin
L^1(\Omega;V\otimes\R^d)
\right\}.
\]
Generative AI has been exploited. The usage is detailed in a specific Section.  
% Let $V,W$ be finite-dimensional Euclidean spaces and let
% \[
%  \A u:=\sum_{j=1}^d A_j\partial_j u,
%  \qquad A_j\in\operatorname{Lin}(V,W),
% \]
% be a first-order homogeneous differential operator with constant coefficients. We assume that $\A$ is $\C$-elliptic, equivalently that its distributional null-space is finite-dimensional. We prove the following endpoint closure principle: if $u\in \BV^{\A}(\Omega)$ and the approximate gradient satisfies $\nabla_{\ap}u\in \mathrm{L}^1(\Omega;V\otimes\R^d)$, then $u\in \BV(\Omega;V)$ and
% \[
%  |Du|(\Omega)
%  \le C_{d,\A}\Bigl(|\A u|(\Omega)
%  +\|\nabla_{\ap}u\|_{\mathrm{L}^1(\Omega)}\Bigr).
% \]
% No regularity assumption on the boundary of $\Omega$ is required. The proof combines the Poincare inequality modulo the finite-dimensional null-space of $\A$, a weak-$(1,1)$ Calderon-Zygmund estimate for the full approximate gradient of compactly supported $\BV^{\A}$ maps, and a finite-dimensional anchoring argument for null-space polynomials. A fixed-scale grid then gives a uniform $\mathrm{L}^1$ bound for the gradients of standard mollifications. We also record that the full-gradient formulation is quantitatively equivalent to controlling only the component of $\nabla_{\ap}u$ lying in the kernel of the linear map induced by $\A$ on first derivatives.
\end{abstract}
\tableofcontents
\section{Introduction}

The classical Korn inequality for $1<p<\infty$ controls the full gradient of a displacement, up to a constant skew-symmetric matrix, by its
symmetric gradient. However, at the endpoint $p=1$, due to Ornstein's non-inequality \cite{ornstein1962non}, this type of estimate generally fails. Several results are available that recover useful Korn-type estimates for $p=1$, for instance by removing a controlled exceptional set or by allowing different rigid motions on different pieces of the domain,
as in \cite{friedrich2018piecewise,cagnetti2022korn}.
In this paper we propose a Korn-type estimate at the endpoint $p=1$, valid without removing any exceptional set, but at the price of including the $L^1$ norm of the full approximate gradient. In particular, this gives a simple characterization of those $\BD$ maps which fail to belong to $\BV$:
they are precisely the ones whose full approximate gradient is not integrable.
Since the mechanism behind the argument is not specific to the symmetric
gradient, we work in the more general setting of functions of bounded
$\mathcal A$-variation, for first-order $\C$-elliptic operators (see \cref{sbsct:Avar}).\\

Let
\[
 \A u=\sum_{j=1}^d A_j\partial_j u
\]
be a first-order homogeneous differential operator with constant
coefficients. Assuming that $\A$ is $\C$-elliptic, we consider the space
$\BV^{\mathcal A}$ of functions whose $\A$-gradient is a finite Radon measure,
see \cite{breit2017traces} or \cref{sbsct:Avar} below. Moreover, by \cite[Lemma~3.1]{GR2019diff} (see also \cref{thm:criticaldiff} below), every
$u\in\BV^{\mathcal A}$ admits an approximate gradient
$\nabla_{\mathrm{ap}}u$ almost everywhere. Our main theorem then reads as follows.
\begin{theorem}\label{thm:main}
Let $\A$ be a $\C$-elliptic first-order homogeneous differential operator.
Then there exists $C_{d,\A}>0$ such that, for every open set
$\Omega\subset\R^d$ and every $u\in\BV^\A(\Omega)$ satisfying
\[
 \nabla_{\ap}u\in L^1(\Omega;V\otimes\R^d),
\]
one has $u\in\BV(\Omega;V)$ and
\[
 |Du|(\Omega)
 \le C_{d,\A}\left(
 |\A u|(\Omega)+\int_\Omega|\nabla_{\ap}u|\,dx
 \right).
\]
\end{theorem}
We emphasize that the constant $C_{d,\A}$ is independent of $\Omega$ and that the above does not require any regularity assumption on $\partial \Omega$. For the sake of a cleaner statement, we formulate the result in terms of
the full approximate gradient. The essential quantity is however only its component in the kernel of the linear map induced by $\A$ on first derivatives, the complementary component being already controlled by $\A u$.\\

The model case is $\A=\mathcal E$, the symmetric gradient. The gap between
$\BD$ and $\BV$ is closely related to the skew-symmetric part of the
approximate gradient, which is not seen by $\mathcal E u$. This is already visible
in the examples constructed in \cite{conti2018special}: they produce
functions in $\mathrm{SBD}$ which fail to belong to $\mathrm{GBV}$ by
building increasingly large infinitesimal rotations while keeping the
symmetric part of the gradient under control. The natural question is then whether the failure of a $\BD$ map to belong to $\BV$ can be entirely attributed to its absolutely continuous part.  Our result shows that this is exactly the case. In fact,
\[
 \BV_{\mathrm{loc}}(\Omega;\R^d)
 =
 \left\{
 u\in\BD_{\mathrm{loc}}(\Omega):
 \operatorname{skew}\nabla_{\ap}u
 \in L^1_{\mathrm{loc}}(\Omega)
 \right\}.
\]

\cref{thm:main} also gives an interesting a posteriori interpretation of $p=1$ Korn inequalities with exceptional sets. Such results recover coercivity either after discarding a suitable region of the domain or after subtracting suitable rigid motions on different pieces. From the present viewpoint, one may heuristically interpret these procedures as ways of isolating the part of the deformation where the full approximate gradient, or equivalently its uncontrolled null-space component, fails to be integrable. \\

The same picture holds for general $\C$-elliptic operators. The measure
$\A u$ controls the component of the first derivative seen by the symbol,
while the remaining component may lie in its kernel. Requiring
$\nabla_{\ap}u\in L^1$ supplies precisely the missing first-order
information. As a consequence, we also obtain the corresponding
characterization of special functions:
\[
 \mathrm{SBV}_{\mathrm{loc}}(\Omega;V)
 =
 \left\{
 u\in\mathrm{SBV}^{\A}_{\mathrm{loc}}(\Omega):
 \nabla_{\ap}u\in
 L^1_{\mathrm{loc}}(\Omega;V\otimes\R^d)
 \right\}.
\]

Indeed, if
\[
u\in \mathrm{SBV}^{\A}(\Omega)
\qquad\text{and}\qquad
\nabla_{\ap}u\in L^1(\Omega;V\otimes\R^d),
\]
then \cref{thm:main} gives $u\in \BV(\Omega;V)$. Applying the estimate locally on arbitrary open subsets $U\subset\subset\Omega$ yields
\[
|Du|(U)\le C_{d,\A}\left(
|\A u|(U)+\int_U|\nabla_{\ap}u|\,dx
\right),
\]
and hence, by outer regularity,
\[
|D^su|\le C_{d,\A}|(\A u)^s|.
\]
Since $u\in \mathrm{SBV}^{\A}(\Omega)$, the singular part $(\A u)^s$ is concentrated on the jump set. Therefore $D^su$ is concentrated on the same rectifiable set. As the Cantor part of a $\BV$ derivative vanishes on every $\sigma$-finite $\mathcal H^{d-1}$-set, it follows that $D^cu=0$, and hence $u\in \mathrm{SBV}(\Omega;V)$. \\

Finally, using the characterization $\frac{\mathrm{d}\A u}{\mathrm{d}\Ld}=A(\nabla_{\mathrm{ap}}u)$ for a linear map $A$, given in \cite[Lemma 3.1]{GR2019diff}, we also have
\[
\mathrm{SBV}^p_{\mathrm{loc}}(\Omega;V)
=
\left\{
u\in \mathrm{SBV}_{\mathrm{loc}}^{\mathcal{A}}(\Omega)
:
\nabla_{\mathrm{ap}} u
\in \LL_{\mathrm{loc}}^p(\Omega;V\otimes\R^d),\ 
\mathcal H^{d-1}(J_u\cap U)<\infty\ \text{for every }U\Subset\Omega
\right\}.
\]

 \subsection{Sketch of the proof}\label{sbsct:proofsketch}
After interaction with generative AI (see also \cref{AI}), the proof turned out to be surprisingly simple. Is based on a localization and mollification argument both occouring at the same scale. Fix $\eps>0$, mollify $u$ at scale $\eps$ and cover the
interior of $\Omega$ by a grid of cubes of size comparable with $\eps$.
On each cube, the Poincar\'e inequality for $\BV^{\A}$ \cite[Theorem~3.2]{breit2017traces} (see also \cref{thm:PoincareA}) selects a
null-space polynomial $a_Q\in\mathrm{Ker}(\A)$ such that
$u-a_Q$ is controlled in $\mathrm{L}^1$ by $|\A u|$ on a slightly larger
ball. We then split locally
\[
 \nabla u_\eps
 =
 \nabla\bigl((u-a_Q)*\varrho_\eps\bigr)
 +
 \nabla(a_Q*\varrho_\eps)
\]
and we aim to find an $\eps$-independent $\mathrm{L}^1$ bound for $ \nabla u_\eps$. \\

The first term is directly controlled by the Poincar\'e estimate. The main point is therefore the second term. More precisely, since $a_Q$ is smooth, the crucial point is to estimate $ \|\nabla a_Q \|_{\mathrm{L}^1}$, since
\[
\|\nabla(a_Q\star \varrho_{\eps})\|_{\mathrm{L}^1}\lesssim \|\nabla a_Q \|_{\mathrm{L}^1}.
\]
This is where the summability of $\nabla_{\mathrm{ap}} u$ enters the argument: the key point is \cref{prop:local}, which, for the same $a_Q$ selected above, by exploiting standard singular integral theory (\cref{prop:globalweak}) provides an estimate of the form 
\[
 \|\nabla_{\ap}u-\nabla a_Q\|_{\mathrm{L}^{1,\infty}}
 \lesssim |\A u|.
\]
Since $\mathrm{Ker}(\A)$ is finite-dimensional, due to \cite[Theorem~2.6]{breit2017traces} (see also \cref{thm:FDN} below) the weak-$\mathrm{L}^1$ quasi-norm and the strong $\mathrm{L}^1$ norm are equivalent on the space of gradients of null-space polynomials. Therefore
\[
\|\nabla a_Q\|_{L^1} \lesssim \|\nabla a_Q\|_{L^{1,\infty}}\lesssim  \|\nabla_{\mathrm{ap}} u - \nabla  a_Q\|_{L^{1,\infty}} +  \|\nabla_{\mathrm{ap}} u\|_{L^{1,\infty}}\lesssim \|\nabla_{\mathrm{ap}} u\|_{L^{1 }} +  |\A u|.
\]
Combining the two estimates and summing over the grid gives an $\eps$-independent bound
\[
 \int_{\Omega'} |\nabla u_\eps|\,dx
 \lesssim
 |\A u|(\Omega)+\|\nabla_{\ap}u\|_{\mathrm{L}^1(\Omega)}
\]
for every $\Omega'\subset\subset\Omega$. The conclusion then follows from
$\BV$ compactness and an exhaustion of $\Omega$.

% \subsection{Acknowledgements}\label{sbsct:acknowledgements} 

\subsection{AI content disclosure}\label{AI}
The problem addressed in this note had been a long-term project of the author. The author's initial approach was based on a nonlocal approximation of the gradient by spherical averages. In the model case of $\BD$, the resulting quantity splits into a part controlled
by the symmetric gradient and a rotational part, which can be estimated by Calder\'on--Zygmund theory for $p>1$. The endpoint strategy was to seek a direct $L^1$ estimate of this rotational term, initially near regular
jump sets, using the geometry of tubular neighbourhoods. Successive interaction with generative AI  brought this strategy to a final complete working form but not in the generality here presented.\\

Generative AI was then used again as an interactive research assistant to discuss, test, and simplify the argument. Through an extended dialogue, the working strategy was progressively dismantled: several of
its technical components turned out not to be essential, and the argument was reduced to a much simpler local mechanism. In particular, the relevant issue was identified as the control of the gradient of the null-space polynomial selected by the Poincar\'e projection \cref{prop:local}. The local weak endpoint
estimate and its combination with finite-dimensional norm equivalence emerged from this iterative discussion rather than being part of the author's initial proof strategy. This observation led to the localization
and mollification argument used in the present paper.\\

The AI was also used to test intermediate arguments, locate relevant references, and assist with the English presentation. The author takes full responsibility for the contents of the paper.\\

The specific AI provider used is not relevant to the mathematical discussion or to the account of the interaction given above.
% Generative AI  was used as an interactive research
% assistant during the development of the proof. The initial strategy, proposed by the author, was based on a localization and mollification argument, with local approximations by elements of the null-space of $\mathcal A$ and a subsequent summation over a decomposition of the domain. Several possible implementations of this strategy were explored
% through an extended dialogue with the AI. During this process, the argument was progressively simplified, and the key local mechanism was identified: for the same null-space polynomial selected by the Poincar\'e projection, a weak-$L^1$ estimate controls $\nabla_{\mathrm{ap}}u-\nabla a$, while finite-dimensionality of $\mathrm{Ker}(\mathcal A)$ upgrades this to the required strong $L^1$ control of $\nabla a$. This observation led to the present proof.\\

% The AI was also used to test intermediate arguments, check scaling and localization details, locate relevant references, and assist with the English presentation. All mathematical statements and arguments were independently checked by the author, who takes full responsibility for the contents of the paper.

\section{Preliminaries and main result}
\subsection{General notation}\label{sbsct:notation} 
Throughout the paper we assume $d\ge2$. The one-dimensional case is immediate. The notation $\Ld$ stands for the $d$-dimensional Lebesgue measure on $\R^d$. Throughout, constants denoted by $C_{d,\A}$ may change from line to line and depend only on the dimension and the operator $\A$. $B_r(x):=x+rB$ where $B$ stands for the unit Euclidean ball of $\R^d$, while $Q_r(x):=x+rQ$ where $Q:=\left[-\sfrac{1}{2},\sfrac{1}{2}\right]^d$ is the $d$-dimensional unit cube with barycenter at $0$. 
For a given measurable set $A$, we denote by $\ca_A(x)$ the characteristic function of $A$, which takes the value $1$ whenever $x\in A$ and $0$ elsewhere. For $v\in V$ and $\xi\in\R^d$, we identify the tensor
$v\otimes\xi\in V\otimes\R^d$ with the linear map
$\R^d\to V$ defined by
\[
    (v\otimes\xi)h=(\xi\cdot h)v.
\]

\subsection{Maps of bounded $\mathcal{A}$-variation}\label{sbsct:Avar}
Let $\mathcal{A}:C^{\infty}(\R^d;V) \to C^{\infty}(\R^d;W)$ be, for some Euclidean spaces $V,W$, a first-order homogeneous differential operator with constant coefficients
\begin{equation}\label{eq:Adef}
 \A u:=\sum_{j=1}^d A_j\partial_j u,
 \qquad A_j\in\operatorname{Lin}(V,W).
\end{equation}
The space $\mathrm{Lin}(V;W)$ denotes the family of all linear maps from $V$ to $W$. The space of functions with \textit{bounded $\A$-variation} (see also \cite{breit2017traces}) is defined by
 \[
\BV^{\A}(\Omega):= \left\{\left. u\in L^1(\Omega;V) \ \right| \ \A u\in \mathcal{M}(\Omega;W)\right\},
\]
where $\mathcal{M}(\Omega;W)$ denotes the space of $W$-valued finite Radon measures on $\Omega$, and the distributional $\A$-gradient is defined by:
\begin{equation}\label{eqn:adjoint}
\int_{\Omega}\varphi \cdot \mathrm{d}\A u:=\int_\Omega \A^*\varphi\cdot u \mathrm{d} x, \ \forall\  \varphi\in C^{\infty}_c(\Omega;W),
\end{equation}
where the $L^2$-adjoint operator $\A^*:C^1(\R^d;W)\rightarrow C^0(\R^d;V)$ is defined in terms of $\{A_j^*\}_{j=1}^n \subset \mathrm{Lin}(W;V)$, satisfying $A_j z \cdot \eta =z \cdot A_j^* \eta$ for all $z\in V,\eta \in W$, by 
\[
\A^*v=-\sum_{j=1}^n A^*_j \partial_jv .
\]
The classical examples are $\A=\nabla$, for which $\BV^{\A}=\BV$ (see \cite{ambrosio2000functions}), and the symmetric gradient $\A=\mathcal{E}=\frac{\nabla + \nabla^t}{2}$, for which $\BV^{\A}=\BD$ (see \cite{ambrosio1997fine}).\\

For $\xi\in\C^d$ and $v\in V+iV$, the \textit{complexified symbol} is defined as
\begin{equation}\label{eqn:symbol}
\aA[\xi]v:=\sum_{j=1}^d \xi_jA_jv.
\end{equation}
\begin{definition}
The operator $\A$ is said to be \textbf{$\R$-elliptic} if
\(
    \aA[\xi]:V\to W
\)
is injective for every $\xi\in\R^d\setminus\{0\}$. The operator $\A$ is said to be \textbf{$\C$-elliptic} if its complexified symbol
\(
    \aA[\xi]:V+iV\to W+iW
\)
is injective for every $\xi\in\C^d\setminus\{0\}$.
\end{definition}
\begin{remark}\label{rmk:CimpliesR}
{\rm Clearly $\C$-ellipticity implies $\R$-ellipticity.}
\end{remark}
The following useful Leibniz rule applies (see also \cite{breit2017traces}).
\begin{lemma}[Cut-off product rule \cite{breit2017traces}]\label{lem:product}
Let $U\subset\R^d$ be open, $v\in \BV^{\A}(U)$, and $\eta\in C_c^1(U)$. Extend $\eta v$ by zero outside $U$. Then $\eta v\in \BV^{\A}(\R^d)$ and
\begin{equation}\label{eq:product}
 \A(\eta v)
 =\eta\,\A v+\aA[\nabla\eta]v\,\Ld
 \qquad\text{in }\R^d.
\end{equation}
Consequently,
\begin{equation}\label{eq:productbound}
 |\A(\eta v)|(\R^d)
 \le \|\eta\|_{L^\infty(U)}|\A v|(U)
 +C_{\A}\|\nabla\eta\|_{L^\infty(U)}\|v\|_{\mathrm{L}^1(U)}.
\end{equation}
\end{lemma}

\begin{proof}
For $\varphi\in C_c^1(\R^d;W)$, using the distributional adjoint $\A^*$ in \eqref{eqn:adjoint} and the Leibniz rule for $\A^*(\eta\varphi)$, a simple computation immediately gives
\[
 \langle \A(\eta v),\varphi\rangle
 =\langle \eta\A v,\varphi\rangle
 +\int_U \varphi\cdot\aA[\nabla\eta]v\,dx.
\]
This is \eqref{eq:product}. Since the finite-dimensional bilinear map $(\xi,v)\mapsto\aA[\xi]v$ satisfies $|\aA[\xi]v|\le C_{\A}|\xi||v|$, by $\C$-ellipticity, estimate \eqref{eq:productbound} follows.
\end{proof}

We recall from \cite{GR2019diff} that, whenever $\mathcal{A}$ is $\mathbb{C}$-elliptic, every $u\in \BV^{\A}(\Omega)$ is approximately differentiable a.e. in $\Omega$. The cited result is stated on $\R^d$; the local version below follows easily by multiplication with a compactly supported cut-off and the Leibniz rule for $\A$ in \cref{lem:product}.
\begin{theorem}[{\cite[Lemma~3.1]{GR2019diff}} ]\label{thm:criticaldiff}
Let $\Omega\subset\R^d$ be open, let $\A$ be an $\R$-elliptic first-order homogeneous differential operator, and let $u\in \BV^{\A}_{\mathrm{loc}}(\Omega)$. Then $u$ is
$L^1$-differentiable at $\Ld$-almost every point of $\Omega$.
In particular, there exists
$\nabla_{\ap}u(x)\in V\otimes\R^d$ such that
\[
    \fint_{B_r(x)}
    |u(y)-u(x)-\nabla_{\ap}u(x)(y-x)|\,dy
    =o(r)
\]
for $\Ld$-almost every $x\in\Omega$.
\end{theorem}

\section{Technical tools}
 
\subsection{Poincarè inequality and characterization of $\C$-ellipticity}\label{sbsct:Poincare} 

For a set $U\subset\R^d$, let $\Pi_U$ denote the $L^2(U;V)$-orthogonal projection onto the finite-dimensional space 
\[
\mathrm{Ker}(\A):= \{ a\in (C^{\infty}_c(\R^d;V))'  \ : \ \A a=0 \}.
\]
As shown in \cite[Section~3.1]{breit2017traces}, $\Pi_U$ extends boundedly to $\mathrm{L}^1(U;V)$; this extension is the one used below.

\begin{theorem}[Poincar\'e inequality in $\BV^{\A}$; {\cite[Theorem~3.2]{breit2017traces}}]\label{thm:PoincareA}
Assume that $\A$ is $\C$-elliptic. There exists $C_{d,\A}>0$ such that for every ball $B_r(x)\subset\R^d$ and every $u\in \BV^{\A}(B_r(x))$,
\begin{equation}\label{eq:PoincareA}
 \inf_{q\in\mathrm{Ker}(\A)}\|u-q\|_{\mathrm{L}^1(B_r(x))}
 \le \|u-\Pi_{B_r(x)} u\|_{\mathrm{L}^1(B_r(x))}
 \le C_{d,\A}\,r\,|\A u|(B_r(x)).
\end{equation}
\end{theorem}
Finally, the following useful characterization will be used to transform the $L^{1,\infty}$ control on rotations into an $L^1$-type control. 
\begin{theorem}[{\cite[Theorem~2.6]{breit2017traces}}]\label{thm:FDN}
Let $\A$ be a first-order homogeneous differential operator with constant coefficients. The following are equivalent:
\begin{enumerate}[label=\textup{(\alph*)}]
 \item $\A$ has finite-dimensional null-space $\mathrm{Ker}(\A)$;
 \item $\A$ is $\C$-elliptic;
 \item there exists $\ell\in\mathbb N$ such that $\mathrm{Ker}(\A)$ is contained in the space of $V$-valued polynomials of degree at most $\ell$.
\end{enumerate}
\end{theorem}

% The scale factor in \eqref{eq:PoincareA} is essential: $\A$ is first order, hence $|\A u|$ has the same physical scaling as a gradient measure.

% \subsection{Critical differentiability and the absolutely continuous density}

% \subsection{A classical weak-$(1,1)$ multiplier theorem}
% \textcolor{red}{
% We shall use the following standard Calderon-Zygmund fact. See, for instance, \cite[Sections~5.2--5.3]{Grafakos2014}, in particular the weak-$(1,1)$ theorem for Calderon-Zygmund singular integrals.}
% \textcolor{red}{\begin{theorem}[Calderon-Zygmund weak type]\label{thm:CZ}
% Let $m\in C^\infty(\R^d\setminus\{0\};\operatorname{Lin}(W,Z))$ be homogeneous of degree zero, where $W,Z$ are finite-dimensional Euclidean spaces. Let $T_m$ be the Fourier multiplier with symbol $m$. Then there exists $C_m>0$ such that
% \begin{equation}\label{eq:CZ}
%  \|T_m f\|_{\mathrm{L}^{1,\infty}(\R^d;Z)}
%  \le C_m\|f\|_{\mathrm{L}^1(\R^d;W)}
% \end{equation}
% for every $f\in \mathrm{L}^1(\R^d;W)$ for which $T_mf$ is initially defined, and hence by the standard weak-type extension for all $f\in \mathrm{L}^1$.
% \end{theorem}}
% \textcolor{red}{Smooth zero-homogeneous multipliers are sums of a bounded multiple of the identity and classical principal-value convolution operators with smooth homogeneous kernels of degree $-d$, so \cref{thm:CZ} is exactly the standard Calderon-Zygmund endpoint estimate.
% }

\subsection{Mollified gradients at differentiability points}\label{sbsct:Molly}

Fix a standard mollifier $\varrho\in C_c^\infty(B_1)$, $\varrho\ge0$, $\int\varrho=1$, and set $\varrho_\delta(x)=\delta^{-d}\varrho(x/\delta)$. Note that, in our notation, for a map $v:\R^d\to V$,
\[
    \nabla(v*\varrho_\delta)(x)
    =
    \int_{\R^d}v(y)\otimes\nabla\varrho_\delta(x-y)\,dy,
\]
since, when applied to $e_j$, the right-hand side equals
$\int_{\R^d}v(y)\partial_j\varrho_\delta(x-y)\,dy
=\partial_j(v*\varrho_\delta)(x)$.

\begin{lemma}[Mollified approximate gradients]\label{lem:mollgrad}
Let $v\in \mathrm{L}^1_{\mathrm{loc}}(\R^d;V)$. Suppose that $v$ is $\mathrm{L}^1$-differentiable at $x$, namely
\[
 \fint_{B_r(x)}|v(y)-v(x)-\nabla_{\mathrm{ap}}v(x)(y-x)|\mathrm{d} y=o(r).
\]
Then
\begin{equation}\label{eq:mollgrad}
 \nabla(v*\varrho_\delta)(x)\longrightarrow \nabla_{\mathrm{ap}}v(x)
 \qquad\text{as }\delta\downarrow0.
\end{equation}
\end{lemma}

\begin{proof}
Set
\[
 R(y):=v(y)-v(x)-\nabla_{\mathrm{ap}}v(x)(y-x).
\]
Then
\begin{align*}
\nabla(v*\varrho_\delta)(x)
&=
\int_{\R^d} v(y)\otimes \nabla\varrho_\delta(x-y)\,dy
\\
&=
\int_{\R^d}
\Bigl(
v(x)+\nabla_{\rm ap}v(x)(y-x)+R(y)
\Bigr)
\otimes \nabla\varrho_\delta(x-y)\,dy
\\
&=
v(x)\otimes
\int_{\R^d}\nabla\varrho_\delta(x-y)\,dy
\\
&\quad
+
\nabla_{\rm ap}v(x)
\int_{\R^d}
(y-x)\otimes\nabla\varrho_\delta(x-y)\,dy
\\
&\quad
+
\int_{\R^d}
R(y)\otimes\nabla\varrho_\delta(x-y)\,dy
\\
&=
\nabla_{\rm ap}v(x)
+
\int_{B_\delta(x)}
R(y)\otimes\nabla\varrho_\delta(x-y)\,dy.
\end{align*}
The last equality follows from integration by parts and the properties of the mollifier. Therefore
% \[
%  \nabla(v*\varrho_\delta)(x)-\nabla_{\mathrm{ap}}v(x)
%  =\int_{B_\delta(x)}R(y)\otimes\nabla\varrho_\delta(x-y)\mathrm{d} y.
% \]

\[
 |\nabla(v*\varrho_\delta)(x)-\nabla_{\mathrm{ap}}v(x)|
 \le C_\varrho\delta^{-d-1}
 \int_{B_\delta(x)}|R(y)|\,dy
 =o(1).
\]
\end{proof}

% \subsection{A local-to-global mollification criterion for $\BV$}

% For $u\in \mathrm{L}^1_{\mathrm{loc}}(\Omega;V)$ set $u_\eps=u*\varrho_\eps$ on
% \[
%  \Omega_\eps:=\{x\in\Omega:\dist(x,\partial\Omega)>\eps\}.
% \]

% \begin{lemma}[Mollification criterion]\label{lem:BVcriterion}
% Let $u\in \mathrm{L}^1_{\mathrm{loc}}(\Omega;V)$. Assume that there exists $M<\infty$ such that for every $\Omega'\Subset\Omega$,
% \begin{equation}\label{eq:BVcriterion}
%  \limsup_{\eps\downarrow0}\int_{\Omega'}|\nabla u_\eps|\,dx\le M.
% \end{equation}
% Then $u\in \BV(\Omega;V)$ and $|Du|(\Omega)\le M$.
% \end{lemma}

% \begin{proof}
% Let $\Phi\in C_c^1(\Omega;V\otimes\R^d)$ with $\|\Phi\|_{L^\infty}\le1$, and choose $\Omega'\Subset\Omega$ containing $\supp\Phi$. For small $\eps$,
% \[
%  \int_\Omega u_\eps\cdot\operatorname{div}\Phi\,dx
%  =-\int_\Omega \nabla u_\eps:\Phi\,dx.
% \]
% Since $u_\eps\to u$ in $\mathrm{L}^1(\Omega')$,
% \[
%  \left|\int_\Omega u\cdot\operatorname{div}\Phi\,dx\right|
%  \le \limsup_{\eps\downarrow0}\int_{\Omega'}|\nabla u_\eps|\,dx
%  \le M.
% \]
% Taking the supremum over all such $\Phi$ gives the total variation of the distributional gradient.
% \end{proof}
\subsection{$\mathrm{L}^{1,\infty}$-control on the approximate gradient}\label{sbsct:weakL1}
For a measurable function $f$ defined on an open set $U\subset \R^d$, we recall that the weak-$\mathrm{L}^1$ quasi-norm is defined by:
\begin{equation}\label{eq:weaknorm}
 \|f\|_{\mathrm{L}^{1,\infty}(U)}
 :=\sup_{s>0}s\,\Ld\bigl(\{x\in U:|f(x)|>s\}\bigr).
\end{equation}
In particular, it is immediate that
\begin{equation}\label{eq:L1toWeak}
 \|f\|_{\mathrm{L}^{1,\infty}(U)}\le \|f\|_{\mathrm{L}^1(U)}.
\end{equation}

The next proposition is a standard consequence of the potential representation for elliptic operators and the weak $(1,1)$ boundedness of Calderón--Zygmund operators. See also \cite{ambrosio1997fine} for the $\BD$ case.

\begin{proposition}[Global weak endpoint estimate]
\label{prop:globalweak}
Assume that $\A$ is elliptic. Then there exists $C_{d,\A}>0$
such that every compactly supported
$v\in BV^{\A}(\R^d)$ satisfies
\begin{equation}\label{eq:globalweak}
    \|\nabla_{\ap}v\|_{L^{1,\infty}(\R^d)}
    \le C_{d,\A}|\A v|(\R^d).
\end{equation}
\end{proposition}
 \begin{proof}
By the potential representation for elliptic operators
\cite[Lemma~2.1]{GR2019diff}, one can write
$v=K_{\A}*\A v$, where $K_{\A}$ is smooth away from the origin
and $(1-d)$-homogeneous. By \cite[Section~3.3, Theorem~3.4]{alberti2014p},
$\nabla_{\ap}v$ is the sum of the singular integral associated with
$\nabla K_{\A}$ and a bounded linear image of
$\frac{d\A v}{d\Ld}$. The former satisfies the classical weak-$(1,1)$
estimate, while the latter belongs to $L^1$ with norm bounded by
$C_{\A}|\A v|(\R^d)$. Hence
\[
\|\nabla_{\ap}v\|_{L^{1,\infty}(\R^d)}
\le C_{d,\A}|\A v|(\R^d).
\]
 \end{proof}

\section{Weak Korn-Poincar\'e estimate and control of the null-space gradient}

We now combine Poincar\'e with \cref{prop:globalweak}. The key point is that the weak estimate is obtained with the same null-space polynomial selected by the Poincar\'e projection.

\begin{proposition}[Local weak Korn--Poincar\'e estimate] \label{prop:local}
Assume that $\A$ is $\C$-elliptic. Let $u\in \BV^{\A}(B_r(x))$. Set
\[
 a_{r,x}:=\Pi_{B_r(x)} u\in\mathrm{Ker}(\A).
\]
Then
\begin{equation}\label{eq:localPoincare}
 \int_{B_r(x)}|u- a_{r,x}|\,dx
 \le C_{d,\A}r|\A u|(B_r(x)),
\end{equation}
and
\begin{equation}\label{eq:localweak}
 \|\nabla_{\ap}u-\nabla  a_{r,x}\|_{\mathrm{L}^{1,\infty}(B_{r/2}(x))}
 \le C_{d,\A}|\A u|(B_r(x)).
\end{equation}
\end{proposition}

\begin{proof}
Estimate \eqref{eq:localPoincare} follows directly from \cref{thm:PoincareA}. Now choose a cutoff function $\eta\in C_c^\infty(B_r(x))$ such that
\[
 0\le\eta\le1,
 \qquad
 \eta\equiv1\text{ on }B_{r/2}(x),
 \qquad
 |\nabla\eta|\le C_dr^{-1}.
\]
Set
\begin{equation}
 w:=\ca_{B_r(x)}\eta(u-a_{r,x}) \in \mathrm{BV}^{\mathcal{A}}(\R^d).
\end{equation}
Since $a_{r,x}\in\mathrm{Ker}(\A)$, one has $\A a_{r,x}=0$. By \cref{lem:product} and \eqref{eq:localPoincare},
\begin{align}
 |\A w|(\R^d)
 &\le |\A u|(B_r(x))
 +C_{\A}r^{-1}\int_{B_r(x)}|u-a_{r,x}|\,dx \le C_{d,\A}|\A u|(B_r(x)).
 \label{eq:Awlocal}
\end{align}
Applying \cref{prop:globalweak} to $w$ gives
\[
 \|\nabla_{\ap}w\|_{\mathrm{L}^{1,\infty}(\R^d)}
 \le C_{d,\A}|\A u|(B_r(x)).
\]
On $B_{r/2}(x)$, $w=u-a_{r,x}$ and hence
\[
 \nabla_{\ap}w=\nabla_{\ap}u-\nabla a_{r,x}
 \qquad\text{a.e.}
\]
Restricting the weak estimate proves \eqref{eq:localweak}.
\end{proof}

Unlike the $\BV$ or $\BD$ case, $a_{r,x}$ need not be affine; however, by \cref{thm:FDN}, it is a polynomial of degree bounded solely in terms of $\A$, and thus we can recover an $L^1$-type control on the rotational part from the $L^{1,\infty}$ control given by \cref{prop:local}.

\begin{lemma}\label{lem:finitedim}
Assume that $\A$ is $\C$-elliptic. There exists $C_{d,\A}>0$ such that for every ball $B_r(x_0)$ and every $a\in\mathrm{Ker}(\A)$,
\begin{equation}\label{eq:finitedim}
 \int_{B_r(x_0)}|\nabla a|\,dx
 \le C_{d,\A}\|\nabla a\|_{\mathrm{L}^{1,\infty}(B_r(x_0))}.
\end{equation}
The constant is independent of $x_0$ and $r$.
\end{lemma}

\begin{proof}
By translation and scaling it is enough to prove the estimate on the unit ball. Indeed, if
\(
    \widetilde a(y):=a(x_0+ry),
\)
then $\widetilde a\in\mathrm{Ker}(\A)$ (since $\A$ is a first-order homogeneous differential operator with constant coefficients) and
\(
    \nabla\widetilde a(y)=r\nabla a(x_0+ry).
\)
By \cref{thm:FDN}, the space
\[
 X:=\{\nabla a|_{B} :a\in\mathrm{Ker}(\A)\}
\]
is finite-dimensional and consists of polynomial tensor fields. Since $\|\cdot \|_{\mathrm{L}^{1,\infty}(B)}$ is a quasi-norm and $\|\cdot\|_{\mathrm{L}^1(B)}$ is a norm (hence, in particular, a quasi-norm), and since all quasi-norms are equivalent on finite-dimensional spaces, we conclude. 
% We claim that on $X$ the $\mathrm{L}^1(B_1)$ norm is controlled by the weak-$\mathrm{L}^1(B_1)$ quasi-norm. Suppose not. Then there exists $g_k\in X$ such that
% \[
%  \|g_k\|_{\mathrm{L}^1(B_1)}=1,
%  \qquad
%  \|g_k\|_{\mathrm{L}^{1,\infty}(B_1)}\to0.
% \]
% Since $X$ is finite-dimensional, after passing to a subsequence $g_k$ converges uniformly on $B_1$ to some $g\in X$. In particular, $\|g\|_{\mathrm{L}^1(B_1)}=1$. Fix $t>0$. Uniform convergence implies that for $k$ large,
% \[
%  \{|g|>2t\}\subset\{|g_k|>t\}.
% \]
% Hence
% \[
%  t\,\Ld(\{|g|>2t\})
%  \le \|g_k\|_{\mathrm{L}^{1,\infty}(B_1)}\to0.
% \]
% Thus $\Ld(\{|g|>2t\})=0$ for every $t>0$, so $g=0$ almost everywhere, a contradiction. This proves \eqref{eq:finitedim} on $B_1$, and scaling gives the general case.
\end{proof}

\begin{lemma}[$L^1$ control of the null-space gradient]\label{lem:anchor}
Under the assumptions of \cref{prop:local},
\begin{equation}\label{eq:anchor}
 \int_{B_{r/2}(x)}|\nabla a_{r,x}|\,dx
 \le C_{d,\A}\left(
 \int_{B_{r/2}(x)}|\nabla_{\ap}u|\,dx
 +|\A u|(B_r(x))
 \right).
\end{equation}
\end{lemma}

\begin{proof}
Observe that, for two arbitrary measurable functions $f,g$, we trivially have
\[
 \{|f+g|>s\}\subset\{|f|>s/2\}\cup\{|g|>s/2\}.
\]
Thus
\begin{equation}\label{eq:weaktriangle}
 \|f+g\|_{\mathrm{L}^{1,\infty}(U)}
 \le2\bigl(\|f\|_{\mathrm{L}^{1,\infty}(U)}+\|g\|_{\mathrm{L}^{1,\infty}(U)}\bigr),
\end{equation}
Using \cref{lem:finitedim} on $B_{r/2}(x)$, then \eqref{eq:weaktriangle}, \eqref{eq:L1toWeak}, and \eqref{eq:localweak}, we obtain
\begin{align*}
 \int_{B_{r/2}(x)}|\nabla a_{r,x}|\,dx
 &\le C_{d,\A}\|\nabla a_{r,x}\|_{\mathrm{L}^{1,\infty}(B_{r/2}(x))}\\
 &\le C_{d,\A}\left(
 \|\nabla_{\ap}u\|_{\mathrm{L}^{1,\infty}(B_{r/2}(x))}
 +\|\nabla_{\ap}u-\nabla a_{r,x}\|_{\mathrm{L}^{1,\infty}(B_{r/2}(x))}
 \right)\\
 &\le C_{d,\A}\left(
 \int_{B_{r/2}(x)}|\nabla_{\ap}u|\,dx
 +|\A u|(B_r(x))
 \right).
\end{align*}
\end{proof}
\section{Proof of the main theorem}
\begin{proof}[Proof of \cref{thm:main}]

Fix $\Omega'\subset\subset\Omega$ and set
\[
 \delta:=\dist(\Omega',\partial\Omega)>0.
\]
Let $\eps>0$ be small and consider the standard grid of pairwise disjoint cubes $Q_{4\eps}(x_Q)$. Let $\mathcal Q_\eps(\Omega')$ be the family of grid cubes meeting $\Omega'$. For each $Q_{4\eps}(x_Q)\in\mathcal Q_\eps(\Omega')$ consider the concentric balls $ B_{6\eps\sqrt d\, }(x_Q)\supset B_{3\eps\sqrt d\, }(x_Q)\supset Q_{6\eps}(x_Q)$ . Note that for $\eps<\eps_0(\delta)$ small enough we can guarantee that $B_{6\eps\sqrt d\,}(x_Q)\subset\subset \Omega$ for every $Q_{4\eps}(x_Q)\in\mathcal Q_\eps(\Omega')$.\\

Now, for each $Q_{4\eps}(x_Q)$, apply \cref{prop:local} on $B_{6\eps\sqrt d\,}(x_Q)$. Thus, setting
\[
 a_{\eps,x_Q}:=\Pi_{B_{6\eps\sqrt d}(x_Q)}u\in\mathrm{Ker}(\A),
\]
we have
\begin{equation}\label{eq:localPoincareQ}
 \int_{B_{6\eps\sqrt d}(x_Q)}
 |u-a_{\eps,x_Q}|\,dx
 \le
 C_{d,\A}\,\eps\,
 |\A u|\bigl(B_{6\eps\sqrt d}(x_Q)\bigr),
\end{equation}
and
\begin{equation}\label{eq:localweakQ}
 \|\nabla_{\ap}u-\nabla a_{\eps,x_Q}\|_
 {\mathrm{L}^{1,\infty}(B_{3\eps\sqrt d}(x_Q))}
 \le
 C_{d,\A}\,
 |\A u|\bigl(B_{6\eps\sqrt d}(x_Q)\bigr).
\end{equation}
Take a mollifier $\{\varrho_{\eps}\}_{\eps>0}$ as in Subsection \ref{sbsct:Molly}. For $x\in\Omega_\eps:=\{x\in\Omega:\dist(x,\partial\Omega)>\eps\}$, set
\[
u_\eps(x):=\int_{B_\eps}u(x-z)\varrho_\eps(z)\,dz.
\]
On $Q_{4\eps}(x_Q)$ we decompose
\begin{equation}\label{eq:griddecomp}
 \nabla u_\eps
 =\nabla\bigl((u-a_{\eps,x_Q})*\varrho_\eps\bigr)
 +\nabla(a_{\eps,x_Q}*\varrho_\eps).
\end{equation}
For the oscillation term, Fubini's theorem gives
\begin{align}
 \int_{Q_{4\eps}(x_Q)}\left|\nabla\bigl((u-a_{ \eps,x_Q})*\varrho_\eps\bigr)(x)\right|\,dx
 &\le \int_{B_\eps}|\nabla\varrho_\eps(z)|
 \int_{Q_{4\eps}(x_Q)}|u(x-z)-a_{ \eps,x_Q}(x-z)|\,dx\,dz\notag\\
 &\le \|\nabla\varrho_\eps\|_{\mathrm{L}^1}
 \int_{Q_{6\eps}(x_Q)}|u-a_{\eps,x_Q}|\,dx\notag\\
 &\le \frac{C_\varrho}{\eps}
 \int_{B_{6\eps\sqrt d\,}(x_Q)}|u-a_{\eps,x_Q}|\,dx\notag\\
 &\le C_{d,\A}|\A u|(B_{6\eps\sqrt d\,}(x_Q)).
 \label{eq:gridosc}
\end{align}
The second inequality follows from the fact that, since $\supp\varrho_\eps\subset B_\eps$, for every $x\in Q_{4\eps}(x_Q)$ and $z\in\supp\varrho_\eps$ one has $x-z\in Q_{6\eps}(x_Q) $. In the last inequality, we applied \eqref{eq:localPoincareQ}.\\

For the null-space term we now use that $a_{\eps,x_Q}$ is smooth. Hence
\[
 \nabla(a_{\eps,x_Q}*\varrho_\eps)=(\nabla a_{\eps,x_Q})*\varrho_\eps,
\]
and, again by Fubini,
\begin{align}
 \int_{Q_{4\eps}(x_Q)}|\nabla(a_{\eps,x_Q}*\varrho_\eps)(x)|\,dx
 &\le \int_{B_\eps}\varrho_\eps(z)
 \int_{Q_{4\eps}(x_Q)}|\nabla a_{\eps,x_Q}(x-z)|\,dx\,dz\notag\\
 &\le \int_{Q_{6\eps}(x_Q)}|\nabla a_{\eps,x_Q}|\,dx
 \le \int_{B_{3\eps\sqrt d\,}(x_Q)}|\nabla a_{\eps,x_Q}|\,dx.
 \label{eq:gridpoly1}
\end{align}

Applying \cref{lem:anchor} on $B_{6\eps\sqrt d\,}(x_Q)$ and using \eqref{eq:gridpoly1} gives
\begin{equation}\label{eq:gridpoly}
 \int_{Q_{4\eps}(x_Q)}|\nabla(a_{\eps ,x_Q}*\varrho_\eps)|\,dx
 \le C_{d,\A}\left(
 \int_{B_{3\eps\sqrt d\,}(x_Q)}|\nabla_{\ap}u|\,dx
 +|\A u|(B_{6\eps\sqrt d\,}(x_Q))
 \right).
\end{equation}
Combining \eqref{eq:griddecomp}, \eqref{eq:gridosc}, and \eqref{eq:gridpoly} yields
\begin{equation}\label{eq:onecubefinal}
 \int_{Q_{4\eps}(x_Q)}|\nabla u_\eps|\,dx
 \le C_{d,\A}\left(
 \int_{B_{3\eps\sqrt d\,}(x_Q)}|\nabla_{\ap}u|\,dx
 +|\A u|(B_{6\eps\sqrt d\,}(x_Q))
 \right).
\end{equation}

The family $\{B_{6\eps\sqrt d\,}(x_Q):Q_{4\eps}(x_Q)\in\mathcal Q_\eps(\Omega')\}$ has overlap bounded by a constant depending only on $d$. The same is true for the half-balls, while the cubes are essentially disjoint. Summing \eqref{eq:onecubefinal} over $Q_{4\eps}(x_Q)\in\mathcal Q_\eps(\Omega')$ yields
\begin{equation}\label{eq:gridglobal}
 \int_{\Omega'}|\nabla u_\eps|\,dx
 \le C_{d,\A}\left(
 \int_\Omega|\nabla_{\ap}u|\,dx
 +|\A u|(\Omega)
 \right)
\end{equation}
for every $\eps<\eps_0(\delta)$ small enough. \\

By the standard properties of mollifiers, $u_\eps\to u$ in $\mathrm{L}^1(\Omega';V)$. At the same time, the compactness theorem in $\BV(\Omega';V)$, together with
\eqref{eq:gridglobal}, yields, up to a subsequence,
$u_\eps\to v$ in $\mathrm{L}^1(\Omega';V)$ for some
$v\in\BV(\Omega';V)$. Hence $u=v$ and therefore
$u\in\BV(\Omega';V)$. Moreover, by lower semicontinuity,
\[
 |Du|(\Omega')
 \le C_{d,\A}\left(
 \int_\Omega|\nabla_{\ap}u|\,dx
 +|\A u|(\Omega)
 \right).
\]
Since $\Omega'\subset\subset\Omega$ is arbitrary and the right-hand side is independent of $\Omega'$, an exhaustion of $\Omega$ by relatively compact open subsets yields $u\in\BV(\Omega;V)$ and
\[
 |Du|(\Omega)
 \le C_{d,\A}\left(
 \int_\Omega|\nabla_{\ap}u|\,dx
 +|\A u|(\Omega)
 \right).
\]
\end{proof}

\bibliography{references}

\begin{thebibliography}{1}

\bibitem{alberti2014p}
Giovanni Alberti, Stefano Bianchini, and Gianluca Crippa.
\newblock On the $ l^p $-differentiability of certain classes of functions.
\newblock {\em Revista Matem{\'a}tica Iberoamericana}, 30(1):349--367, 2014.

\bibitem{ambrosio1997fine}
Luigi Ambrosio, Alessandra Coscia, and Gianni Dal~Maso.
\newblock Fine properties of functions with bounded deformation.
\newblock {\em Archive for Rational Mechanics and Analysis}, 139(3):201--238,
  1997.

\bibitem{ambrosio2000functions}
Luigi Ambrosio, Nicola Fusco, and Diego Pallara.
\newblock {\em Functions of bounded variation and free discontinuity problems},
  volume 254.
\newblock Clarendon Press Oxford, 2000.

\bibitem{breit2017traces}
Dominic Breit, Lars Diening, and Franz Gmeineder.
\newblock {On the trace operator for functions of bounded
  $\mathbb{A}$-variation}.
\newblock {\em Analysis e PDE}, 13(2):559 -- 594, 2020.

\bibitem{cagnetti2022korn}
Filippo Cagnetti, Antonin Chambolle, and Lucia Scardia.
\newblock Korn and poincar{\'e}-korn inequalities for functions with a small
  jump set.
\newblock {\em Mathematische Annalen}, 383(3--4):1179--1216, 2022.

\bibitem{conti2018special}
Sergio Conti, Matteo Focardi, and Flaviana Iurlano.
\newblock Which special functions of bounded deformation have bounded
  variation?
\newblock {\em Proceedings of the Royal Society of Edinburgh Section A:
  Mathematics}, 148(1):33--50, 2018.

\bibitem{friedrich2018piecewise}
Manuel Friedrich.
\newblock A piecewise korn inequality in sbd and applications to embedding and
  density results.
\newblock {\em SIAM Journal on Mathematical Analysis}, 50(4):3842--3918, 2018.

\bibitem{GR2019diff}
Franz Gmeineder and Bogdan Rai{\c{t}}{\u{a}}.
\newblock On critical $l^p$-differentiability of $bd$-maps.
\newblock {\em Revista Matem{\'a}tica Iberoamericana}, 35(7):2071--2078, 2019.

\bibitem{ornstein1962non}
Donald Ornstein.
\newblock A non-inequality for differential operators in the $l^1$ norm.
\newblock {\em Archive for Rational Mechanics and Analysis}, 11(1):40--49,
  1962.

\end{thebibliography}
\bibliographystyle{plain}

\end{document}